\documentclass[12pt]{amsart}

\usepackage{amssymb}
\usepackage{bm}
\usepackage{graphicx}
\usepackage{psfrag}
\usepackage{hyperref}
\usepackage{amscd}
\usepackage{color}
\usepackage{url}
\usepackage{algpseudocode}
\usepackage{fancyhdr}
\usepackage{xy}
\input xy
\xyoption{all}

\newtheorem*{maintheorem*}{Main Theorem}
\newtheorem{theorem}{Theorem}[section]
\newtheorem{proposition}[theorem]{Proposition}
\newtheorem{lemma}[theorem]{Lemma}
\newtheorem{corollary}[theorem]{Corollary}

\theoremstyle{definition}
\newtheorem{definition}[theorem]{Definition}
\newtheorem{remark}[theorem]{Remark}
\newtheorem{example}[theorem]{Example}
\newtheorem{exercise}[theorem]{Exercise}

\numberwithin{equation}{section}

\newcommand{\cI}{\mathcal{I}}
\def\z{\mathbb{Z}}
\def\fz{\z[\sqrt{-5}]}
\def\fq{\q(\sqrt{-5})}

 \newcommand{\C}{{\mathbb{C}}}

  \newfont{\eu}{eurb10 at 11pt}
 \newfont{\eus}{eusb10 at 11pt}

\newcommand \q{{\mathbb Q}}
\newcommand \n{{\mathbb N}}

\hypersetup{
	pdftitle={How Do Elements Really Factor in Rings of Integers},
	colorlinks=true,
	linkcolor=blue,
	citecolor=cyan,
	urlcolor=wine
}


\begin{document}
	
	\mbox{}
	\title{How do elements really factor in $\mathbb{Z}[\sqrt{-5}]$?}
	
	\author{Scott T. Chapman}
	\address{Department of Mathematics and Statistics\\Sam Houston State University\\Huntsville, TX  77341}
	\email{scott.chapman@shsu.edu}
	
	\author{Felix Gotti}
	\address{Department of Mathematics\\UC Berkeley\\Berkeley, CA 94720 \newline \indent Department of Mathematics\\Harvard University\\Cambridge, MA 02138}
	\email{felixgotti@berkeley.edu}
	\email{felixgotti@harvard.edu}
	%\address{Department of Mathematics\\UC Berkeley\\Berkeley, CA 94720}
	%\email{felixgotti@berkeley.edu}
	
	\author{Marly Gotti}
	\address{Department of Mathematics\\University of Florida\\Gainesville, FL 32611}
	\email{marlycormar@ufl.edu}
	
	
	\date{\today}

\begin{abstract}
	Most undergraduate level abstract algebra texts use $\fz$ as an example of an integral domain which is not a unique factorization domain (or UFD) by exhibiting two distinct irreducible factorizations of a nonzero element. But such a brief example, which requires merely an understanding of basic norms, only scratches the surface of how elements actually factor in this ring of algebraic integers. We offer here an interactive framework which shows that while $\fz$ is not a UFD, it does satisfy a slightly weaker factorization condition, known as half-factoriality. The arguments involved revolve around the Fundamental Theorem of Ideal Theory in algebraic number fields.
	\bigskip
	\begin{center}
		\textit{Dedicated to David F. Anderson on the occasion of his retirement.}
	\end{center}
\end{abstract}

\maketitle

\tableofcontents

\section{Introduction}
\label{sec:intro}

Consider the integral domain
\[
	\fz = \{ a+b\sqrt{-5} \mid a,b \in \z\}.
\]
Your undergraduate abstract algebra text probably used it as the base example of an integral domain that is not a unique factorization domain (or UFD). The Fundamental Theorem of Arithmetic fails in $\fz$ as this domain contains elements with multiple factorizations into irreducibles; for example,
\begin{equation}\label{2ndzroot5}
	6=2\cdot 3 = (1 - \sqrt{-5})(1 + \sqrt{-5})
\end{equation}
even though $2,3, 1 - \sqrt{-5}$, and $1+ \sqrt{-5}$ are pairwise non-associate irreducible elements in $\fz$. To argue this, the norm on $\fz$, i.e.,
\begin{align}
	N(a+b\sqrt{-5})=a^2+5b^2, \label{eq:standard norm}
\end{align}
plays an important role, as it is a multiplicative function satisfying the following properties:
\begin{itemize}
	\item $N(\alpha)=0$ if and only if $\alpha=0$;
	\vspace{2pt}
	\item $N(\alpha\beta)=N(\alpha)N(\beta)$ for all $\alpha ,\beta\in \fz$;
	\vspace{2pt}
	\item $\alpha$ is a unit if and only if $N(\alpha) = 1$ (i.e., $\pm 1$ are the only units of $\fz$);
	\vspace{2pt}
	\item if $N(\alpha)$ is prime, then $\alpha$ is irreducible.
\end{itemize}
However, introductory abstract algebra books seldom dig deeper than what Equation \eqref{2ndzroot5} does. The goal of this paper is to use ideal theory to describe exactly how elements in $\fz$ factor into products of irreducibles.  In doing so, we will show that $\fz$ satisfies a nice factorization property, which is known as \textit{half-factoriality}. Thus, we say that $\fz$ is a \emph{half-factorial domain} (or HFD).  Our journey will require nothing more than elementary algebra, but will give the reader a glimpse of how The Fundamental Theorem of Ideal Theory resolves the non-unique factorizations of $\fz$. The notion that unique factorization in rings of integers could be recovered via ideals was important in the late 1800's in attempts to prove Fermat's Last Theorem (see \cite[Chapter~11]{PD}).

Our presentation is somewhat interactive, as many steps that follow from standard techniques of basic algebra are left to the reader as exercises. The only background we expect from the reader are introductory courses in linear algebra and abstract algebra. Assuming such prerequisites, we have tried to present here a self-contained and friendly approach to the phenomenon of non-uniqueness of factorizations occurring in $\fz$.  More advanced and general arguments (which apply to any ring of integers) can be found in~\cite{Ma} and~\cite{PD}.
\medskip

\section{Integral Bases and Discriminants} \label{sec:integral bases}

Although in this paper we are primarily concerned with the phenomenon of non-unique factorizations in the particular ring of integers $\fz$, it is more enlightening from an algebraic perspective to introduce our needed concepts for arbitrary commutative rings with identity, rings of integers, or quadratic rings of integers, depending on the most appropriate context for each concept being introduced. In what follows, we shall proceed in this manner while trying, by all means, to keep the exposition as elementary as possible.

An element $\alpha \in \C$ is said to be \emph{algebraic} provided that it is a root of a nonzero polynomial with rational coefficients, while $\alpha$ is said to be an \emph{algebraic integer} provided that it is a root of a monic polynomial with integer coefficients. It is not hard to argue that every subfield of $\C$ contains $\q$ and is a $\q$-vector space.

\begin{definition}
	A subfield $K$ of $\C$ is called an \emph{algebraic number field} provided that it has finite dimension as a vector space over $\q$. The subset
	\[
		\mathcal{O}_K := \{\alpha \in K \mid \alpha \ \text{ is an algebraic integer}\}
	\]
	of $K$ is called the \emph{ring of integers} of $K$.
\end{definition}

The ring of integers of any algebraic number field is, indeed, a ring. The reader is invited to verify this observation. If $\alpha$ is a complex number, then $\q(\alpha)$ denotes the smallest subfield of $\C$ containing $\alpha$. It is well known that a subfield $K$ of $\C$ is an algebraic number field if and only if there exists an algebraic number $\alpha \in \C$ such that $K = \q(\alpha)$ (see, for example, \cite[Theorem~2.17]{Ja}). Among all algebraic number fields, we are primarily interested in those that are two-dimensional vector spaces over~$\q$.

\begin{definition}
	An algebraic number field that is a two-dimensional vector space over $\q$ is called a \emph{quadratic number field}. If $K$ is a quadratic number field, then $\mathcal{O}_K$ is called a \emph{quadratic ring of integers}.
\end{definition}

For $\alpha \in \C$, let $\z[\alpha]$ denote the set of all polynomial expressions in $\alpha$ having integer coefficients. Clearly, $\z[\alpha]$ is a subring of $\q(\alpha)$. It is also clear that, for $d \in \z$, the field $\q(\sqrt{d})$ has dimension at most two as a
$\q$-vector space and, therefore, it is an algebraic number field. Moreover, if $d \notin \{0,1\}$ and $d$ is squarefree (i.e., $d$ is not divisible by the square of any prime), then it immediately follows that $\q(\sqrt{d})$ is a two-dimensional vector space over $\q$ and, as a result, a quadratic number field. As we are mainly interested in the case when $d = -5$, we propose the following exercise.

\begin{exercise} \label{ex:rings of integers of quadratic fields}
	Let $d \in \z \setminus \{0,1\}$ be a squarefree integer such that $d \equiv 2, 3 \pmod{4}$. Prove that $\z[\sqrt{d}]$ is the ring of integers of the quadratic number field $\q(\sqrt{d})$.
\end{exercise}

\begin{remark}
	When $d \in \z \setminus \{0,1\}$ is a squarefree integer satisfying $d \equiv 1 \pmod{4}$, it is not hard to argue that the ring of integers of $\q(\sqrt{d})$ is $\z[\frac{1 + \sqrt{d}}{2}]$. However, we will not be concerned with this case as our case of interest is $d = -5$.
\end{remark}

For $d$ as specified in Exercise~\ref{ex:rings of integers of quadratic fields}, the elements of $\z[\sqrt{d}]$ can be written in the form $a + b \sqrt{d}$ for $a,b \in \z$. The \emph{norm} $N$ on $\z[\sqrt{d}]$ is defined by
\[
	N(a + b\sqrt{d}) = a^2 - d b^2
\]
(cf. Equation~(\ref{eq:standard norm})). The norm $N$ on $\z[\sqrt{d}]$ also satisfies the four properties listed in the introduction.

Let us now take a look at the structure of an algebraic number field $K$ with linear algebra in mind. %Note that since $\sqrt{-5}$ is algebraic over $\mathbb{Q}$,
%\[
%	\fq = \mathbb{Q}[\sqrt{-5}]=\{r_1+r_2\sqrt{-5}\,\mid\, r_1, r_2\in \mathbb{Q}\}.
%\]
%Moreover, $\fq$ is a two-dimensional vector space over $\mathbb{Q}$. We make the following definition.
For $\alpha \in K$ consider the function $m_\alpha \colon K \to K$ defined via multiplication by $\alpha$, i.e., $m_\alpha(x) = \alpha x$ for all $x \in K$. One can easily see that $m_\alpha$ is a linear transformation of $\q$-vector spaces. Therefore, after fixing a basis for the $\q$-vector space $K$, we can represent $m_\alpha$ by a matrix $M$. The \emph{trace} of $\alpha$, which is denoted by $\text{Tr}(\alpha)$, is defined to be the trace of the matrix $M$. It is worth noting that $\text{Tr}(\alpha)$ does not depend on the chosen basis for $K$. Also, notice that $\text{Tr}(\alpha) \in \q$. Furthermore, if $\alpha \in \mathcal{O}_K$, then $\text{Tr}(\alpha) \in \z$ (see \cite[Lemma~4.1.1]{RE}, or Exercise~\ref{ex:trace and discriminant} for the case when $K = \q(\sqrt{d})$).

\begin{definition}
	Let $K$ be an algebraic number field that has dimension $n$ as a $\q$-vector space. The \emph{discriminant} of a subset $\{\omega_1, \dots, \omega_n\}$ of $K$, which is denoted by $\Delta[\omega_1, \dots, \omega_n]$, is $\det \, T$, where $T$ is the $n \times n$ matrix $\big(\text{Tr}(\omega_i \omega_j)\big)_{1 \le i,j \le n}$.
\end{definition}

With $K$ as introduced above, if $\{\omega_1, \dots, \omega_n\}$ is a subset of $\mathcal{O}_K$, then it follows that $\Delta[\omega_1, \dots, \omega_n] \in \z$ (see Exercise~\ref{ex:trace and discriminant} for the case when $K = \q(\sqrt{d})$). In addition, the discriminant of any basis for the $\q$-vector space $K$ is nonzero; we will prove this for $K = \q(\sqrt{d})$ in Proposition~\ref{prop:discrimvalue}.

\begin{exercise} \label{ex:trace and discriminant}
	Let $d \in \z \setminus \{0,1\}$ be a squarefree integer such that $d \equiv 2,3 \pmod{4}$.
	\begin{enumerate}
		\item If $\alpha= a_1 + a_2\sqrt{d} \in \z[\sqrt{d}]$, then $\text{Tr}(\alpha) = 2a_1$.
		\vspace{3pt}
		\item If, in addition, $\beta = b_1 + b_2 \sqrt{d} \in \z[\sqrt{d}]$, then
		\[
			\Delta[\alpha,\beta] = \left( \det \left[ \begin{matrix} \alpha \, & \sigma(\alpha) \\
			\beta \, & \sigma(\beta) \end{matrix}\right] \right)^2 = 4d(a_1b_2 - a_2b_1)^2,
		\]
		where $\sigma(x + y\sqrt{d}) = x - y \sqrt{d}$ for all $x,y \in \z$.
%		\vspace{3pt}
%		\item If $\alpha, \beta \in \fz$, then
%		\[
%			\Delta[\alpha,\beta] = \left( \det \left[ \begin{matrix} \alpha & \overline{\alpha} \\
%			\beta & \overline{\beta} \end{matrix}\right] \right)^2 = (\alpha\overline{\beta} - \overline{\alpha}\beta)^2,
%		\]
%		where $\overline{x}$ denotes the conjugate of $x$ as a complex number.
	\end{enumerate}
\end{exercise}
\medskip

\begin{example} \label{ex:discriminant of the canonical integral basis}
	Let $d \notin \{0,1\}$ be a squarefree integer such that $d \equiv 2,3 \pmod{4}$. It follows from Exercise~\ref{ex:trace and discriminant} that the subset $\{1, \sqrt{d}\}$ of the ring of integers $\z[\sqrt{d}]$ satisfies that $\Delta[1, \sqrt{d}] =4d$.
%	
%	Moreover, if $a + b\sqrt{-5}$ and $c + d\sqrt{-5}$ are any two elements of $\fz$, then a simple computation yields 
%	\[
%		\Delta[ a+b\sqrt{-5}, c+d\sqrt{-5}]= -20(bc - ad)^2.
%	\]
%	Thus, the discriminant of any two elements in $\fz$ is an element of $\z$.
\end{example}
\medskip

%By the end of this section, we will see that if the pair $\{a+b\sqrt{-5}, c+d\sqrt{-5}\}$ is an integral basis for $\fz$, then $bc-ad$ in Example~\ref{ex:discriminant of the canonical integral basis} is always 1.

We proceed to introduce the concept of integral basis.

\begin{definition}
	Let $K$ be an algebraic number field of dimension $n$ as a vector space over $\q$. The elements $\omega_1, \dots , \omega_n \in \mathcal{O}_K$ form an \emph{integral basis} for $\mathcal{O}_K$ if for each $\beta \in \mathcal{O}_K$ there are unique $z_1,\dots ,z_n \in \z$ satisfying $\beta = z_1 \omega_1 + \dots + z_n \omega_n$.
\end{definition}

\begin{example} \label{ex:canonical integral basis}
	Let $d \in \z \setminus \{0,1\}$ be a squarefree integer such that $d \equiv 2,3 \pmod{4}$. Clearly, every element in $\z[\sqrt{d}]$ is an integral linear combination of $1$ and $\sqrt{d}$. Suppose, on the other hand, that $a_1 + a_2 \sqrt{d} = b_1 + b_2 \sqrt{d}$ for some $a_1, a_2, b_1, b_2 \in \z$. Note that $a_2 = b_2$; otherwise $\sqrt{d} = \frac{a_1 - b_1}{b_2 - a_2}$ would be a rational number. As a result, $a_1 = b_1$. Thus, we have verified that every element of $\z[\sqrt{d}]$ can be uniquely written as an integral linear combination of $1$ and $\sqrt{d}$. Hence, the set $\{1, \sqrt{d}\}$ is an integral basis for $\z[\sqrt{d}]$.
\end{example}

In general, the ring of integers of any algebraic number field has an integral basis (see~\cite[Theorem~3.27]{Ja}). On the other hand, although integral bases are not unique, any two integral bases for the same ring of integers have the same discriminant. We shall prove this for $\z[\sqrt{d}]$ in Theorem~\ref{thm:discrimvalue}. \\

\noindent {\bf Notation:} If $S$ is a subset of the complex numbers, then we let $S^\bullet$ denote $S \! \setminus \! \{0\}$.
\medskip

\begin{lemma}\label{intbasis}
	Let $K$ be an algebraic number field of dimension $n$ as a $\q$-vector space. An integral basis for $\mathcal{O}_K$ is a basis for $K$ as a vector space over $\q$. %Hence, every integral basis for $\fz$ contains exactly two nonzero elements.
\end{lemma}

\begin{proof}
	Suppose that $\{\omega_1, \dots, \omega_n\}$ is an integral basis for $\mathcal{O}_K$, and take rational coefficients $q_1, \dots, q_n$ such that
	\[
		q_1 \omega_1 + \dots + q_n \omega_n = 0.
	\]
	Multiplying the above equality by the common denominator of the nonzero $q_i$'s and using the fact that $\{\omega_1, \dots, \omega_n\}$ is an integral basis for the ring of integers $\mathcal{O}_K$, we obtain that $q_1 = \dots = q_n = 0$. Hence, $\{\omega_1, \dots, \omega_n\}$ is a linearly independent set of the $\q$-vector space $K$. As $K$ has dimension $n$ over $\q$, the set $\{\omega_1, \dots, \omega_n\}$ is a basis for the vector space $K$ over $\q$.
%	
%	the two-dimensional vector space $\fq$. In particular, $k \le 2$. Moreover, for $r_1 + r_2\sqrt{-5} \in \q(\sqrt{-5})$, one can take $n_1, n_2 \in \z^\bullet$ such that $n_1 r_1, n_2 r_2 \in \z$. As $\{\alpha_1, \alpha_2\}$ is an integral basis for $\fz$, there exist $z_1,z_1', z_2, z_2' \in \z$ satisfying $n_1 r_1 = z_1 \alpha_1 + z_2 \alpha_2$ and $n_2 r_2 \sqrt{-5} = z'_1 \alpha_1 + z'_2 \alpha_2$. As a result,
%	\[
%		r_1 \! + \! r_2 \sqrt{-5} = \frac{z_1 \alpha_1 + z_2 \alpha_2}{n_1} + \frac{z_1' \alpha_1 + z_2' \alpha_2}{n_2} = \! \bigg(\frac{z_1}{n_1} + \frac{z_1'}{n_2}\bigg) \alpha_1 \! + \! \bigg(\frac{z_2}{n_1} + \frac{z_2'}{n_2}\bigg) \alpha_2.
%	\]
%	Therefore the set $\{\alpha_1, \alpha_2\}$ spans $\fq$ over $\q$. This implies that the integral basis $\{\alpha_1, \alpha_2\}$ is, indeed, a basis for the vector space $\fq$, which completes the proof. The second statement of the lemma follows immediately.
\end{proof}
\medskip

%Let us prove now that every basis for the vector space $\q(\sqrt{d})$ over $\q$ has nonzero discriminant.

\begin{proposition} \label{prop:discrimvalue}
	Let $d \in \z \setminus \{0,1\}$ be a squarefree integer with $d \equiv 2,3 \pmod{4}$. If $\{\alpha_1,\alpha_2\}$ is a vector space basis for $\q(\sqrt{d})$ contained in $\z[\sqrt{d}]$, then $\Delta[\alpha_1,\alpha_2] \in \z^\bullet$.
\end{proposition}

\begin{proof}
	From the fact that $\{\alpha_1, \alpha_2\} \subseteq \z[\sqrt{d}]$, it follows that $\Delta[\alpha_1, \alpha_2] \in \z$. So suppose, by way of contradiction, that $\Delta[\alpha_1, \alpha_2] = 0$. Taking $\{\omega_1, \omega_2\}$ to be an integral basis for $\z[\sqrt{d}]$, one has that
	\[
		\begin{array} {lclcl}
			\alpha_1 & = &  z_{1,1}\omega_1 & + & z_{1,2}\omega_2 \\
			\alpha_2 & = & z_{2,1}\omega_1 & + & z_{2,2}\omega_2,
		\end{array}
	\]
	for some $z_{i,j} \in \z$. Using Exercise~\ref{ex:trace and discriminant}, we obtain
	\begin{multline}\label{discrim}
		\Delta[\alpha_1,\alpha_2] = \left(\det \left[ \begin{matrix} \alpha_1 & \sigma(\alpha_1) \\  \alpha_2 & \sigma(\alpha_2) \end{matrix}\right]\right)^2 =
		\left(\det \left( \left[\begin{matrix} z_{1,1} & z_{1,2} \\  z_{2,1} & z_{2,2} \end{matrix}\right]
		\left[ \begin{matrix} \omega_1 & \sigma(\omega_1) \\  \omega_2 & \sigma(\omega_2) \end{matrix}\right]\right)\right)^2  \\
		= \left(\det \left[\begin{matrix} z_{1,1} & z_{1,2} \\  z_{2,1} & z_{2,2} \end{matrix}\right]\right)^2
		\left(\det \left[ \begin{matrix} \omega_1 & \sigma(\omega_1) \\  \omega_2 & \sigma(\omega_2) \end{matrix}\right]\right)^2 =
		\left(\det \left[\begin{matrix} z_{1,1} & z_{1,2} \\  z_{2,1} & z_{2,2} \end{matrix}\right]\right)^2\Delta[\omega_1,\omega_2],
	\end{multline}
	where $\sigma(x + y\sqrt{d}) = x - y \sqrt{d}$ for all $x,y \in \z$. If $\omega_1 = 1$ and $\omega_2 = \sqrt{d}$, then
	\[
		\det \left[\begin{matrix} z_{1,1} & z_{2,1} \\  z_{1,2} & z_{2,2} \end{matrix}\right]=
		\det \left[\begin{matrix} z_{1,1} & z_{1,2} \\  z_{2,1} & z_{2,2} \end{matrix}\right]=0,
	\] 
	and so there are elements $q_1, q_2\in \q$ not both zero with 
	\[
		\left[\begin{matrix} z_{1,1} & z_{2,1} \\  z_{1,2} & z_{2,2} \end{matrix}\right]
		\left[\begin{matrix} q_1 \\ q_2 \end{matrix}\right] = \left[\begin{matrix} 0 \\ 0 \end{matrix}\right].
	\]
	Hence,
	\begin{align*}
		0 &= \omega_1(q_1 z_{1,1}+q_2 z_{2,1})+\omega_2(q_1 z_{1,2}+q_2 z_{2,2}) \\
		&= q_1(z_{1,1}\omega_1 + z_{1,2}\omega_2) + q_2(z_{2,1}\omega_1 + z_{2,2}\omega_2) \\
		&= q_1\alpha_1 + q_2\alpha_2,
	\end{align*}
	which is a contradiction because the set $\{\alpha_1, \alpha_2\}$ is linearly independent in the vector space $\q(\sqrt{d})$. Thus, $\Delta[\alpha_1,\alpha_2]\neq 0$, as desired.
\end{proof}
\medskip

\begin{exercise}
	Let $d \in \z \setminus \{0,1\}$ be a squarefree integer with $d \equiv 2,3 \pmod{4}$. Show that $\Delta[\alpha_1,\alpha_2] \neq 0$ whenever $\{\alpha_1,\alpha_2\}$ is a basis for the $\q$-vector space $\q(\sqrt{d})$.
\end{exercise}

%Distinct integral bases for $\fz$ do have something in common, which we will demonstrate in our next theorem.

%\begin{definition}
%	If $\alpha,\beta \in \fz$, then the \textit{discriminant} of $\alpha, \beta$, denoted $\Delta[\alpha, \beta]$, is defined by
%	\[
%		\Delta[\alpha,\beta] = \left( \det \left[ \begin{matrix} \alpha & \overline{\alpha} \\
%		\beta & \overline{\beta} \end{matrix}\right] \right)^2 = (\alpha\overline{\beta} - \overline{\alpha}\beta)^2,
%	\]
%	where $\overline{x}$ denotes the conjugate of $x$ as a complex number.
%\end{definition}
%\medskip

Using Lemma~\ref{intbasis} and Exercise~\ref{ex:trace and discriminant}, we obtain the following important result.

\begin{corollary} \label{cor:discrimvalue}
	Let $d \in \z \setminus \{0,1\}$ be a squarefree integer with $d \equiv 2,3 \pmod{4}$. The discriminant of each integral basis for $\z[\sqrt{d}]$ is in $\z^\bullet$.
\end{corollary}
\medskip

\noindent {\bf Notation:} Let $\n$ denote the set of positive integers, and set $\n_0 = \{0\} \cup \n$.
\medskip

\begin{theorem} \label{thm:discrimvalue}
	Let $d \in \z \setminus \{0,1\}$ be a squarefree integer with $d \equiv 2,3 \pmod{4}$. Any two integral bases for $\z[\sqrt{d}]$ have the same discriminant.
\end{theorem}

\begin{proof}
	Let $\{\alpha_1, \alpha_2\}$ and $\{\omega_1, \omega_2\}$ be integral bases for $\z[\sqrt{d}]$, and let $z_{i,j}$ be defined as in the proof of Proposition~\ref{prop:discrimvalue}. Since $\Delta[\alpha_1, \alpha_2]$ and $\Delta[\omega_1, \omega_2]$ are both integers, Equation~\eqref{discrim} in the proof of Proposition~\ref{prop:discrimvalue}, along with the fact that $\left(\det \left[\begin{matrix} z_{1,1} & z_{1,2} \\  z_{2,1} & z_{2,2} \end{matrix}\right]\right)^2$ belongs to $\n$, implies that $\Delta[\omega_1,\omega_2]$ divides $\Delta[\alpha_1,\alpha_2]$. Using a similar argument, we can show that $\Delta[\alpha_1,\alpha_2]$ divides $\Delta[\omega_1,\omega_2]$.  As both discriminants have the same sign, $\Delta[\alpha_1,\alpha_2]= \Delta[\omega_1,\omega_2]$.
\end{proof}
\medskip

Using Example~\ref{ex:discriminant of the canonical integral basis} and Example~\ref{ex:canonical integral basis}, we obtain the following corollary.

\begin{corollary}
	Let $d \in \z \setminus \{0,1\}$ be a squarefree integer with $d \equiv 2,3 \pmod{4}$.
	Every integral basis for $\z[\sqrt{d}]$ has discriminant $4d$.
\end{corollary}
\medskip

\section{General Properties of Ideals} \label{sec:general properties of ideals}

Let $R$ be a commutative ring with identity.  In most beginning algebra classes, the units, irreducibles, and associate elements in $R$ are standard concepts of interest. Recall that the units of $R$ are precisely the invertible elements, while nonunit elements $x, y \in R$ are associates if $a=ub$ for a unit $u$ of $R$.
A nonunit $x \in R^\bullet := R \setminus \{0\}$ is irreducible if whenever $x = uv$ in $R$, then either $u$ or $v$ is a unit. 

To truly understand factorizations in $\fz$, we will need to know first how ideals of $\fz$ are generated. Recall that a subset $I$ of a commutative ring $R$ with identity is called an ideal of $R$ provided that $I$ is a subring with the property that $rI \subseteq I$ for all $r \in R$. It follows immediately that if $x_1, \dots, x_k \in R$, then the set
\[
	I = \langle x_1, \dots, x_k \rangle=\{r_1 x_1 + \dots + r_k x_k \mid \mbox{ each } r_i \in R\}
\]
is an ideal of $R$, that is, the ideal generated by $x_1, \dots, x_k$. Recall that $I$ is said to be principal if $I = \langle x \rangle$ for some $x \in R$, and $R$ is said to be a principal ideal domain (or a~PID) if each ideal of $R$ is principal. The zero ideal $\langle 0 \rangle$ and the entire ring $R = \langle 1 \rangle$ are principal ideals. May it be that all the ideals of $\fz$ are principal?  It turns out that the answer is ``no'' as we shall see in the next example.

\begin{example}\label{ex:a nonprincipal ideal}
	The ring of integers $\fz$ is not a PID. We argue that the ideal 
	\[
		I = \langle 2, 1+\sqrt{-5}\rangle
	\]
	is not principal. If $I=\langle \alpha \rangle$, then $\alpha$ divides both $2$ and $1+\sqrt{-5}$. The reader will verify in Exercise~\ref{ex:irreducible and non-associate} below that both of these elements are irreducible and non-associates. Hence, $\alpha=\pm 1$ and $I=\langle \pm 1\rangle = \fz$. Now we show that $3 \notin I$. Suppose there exist $a, b, c, d \in \z$ so that
	\[
		(a+b\sqrt{-5})2 + (c+d\sqrt{-5})(1+\sqrt{-5}) = 3.
	\]
	Expanding the previous equality, we obtain
	\begin{equation}\label{idealarray}
		\begin{array}{lclclcl}
			2a & + & c & - & 5d & = & 3 \\
			2b & + & c & + & d & = & 0.
		\end{array}
	\end{equation}
	After subtracting, we are left with $2(a-b)-6d=3$, which implies that $2$ divides $3$ in $\z$, a contradiction.
\end{example}
\medskip

\begin{exercise} \label{ex:irreducible and non-associate}
	Show that the elements $2$ and $1+\sqrt{-5}$ are irreducible and non-associates in $\fz$.  (Hint: use the norm function.)
\end{exercise}
\medskip

Let us recall that a proper ideal $I$ of a commutative ring $R$ with identity is said to be prime if whenever $xy\in I$ for $x,y \in R$, then either $x\in I$ or $y\in I$. In addition, we know that an element $p \in R \setminus \{0\}$ is said to be prime provided that the principal ideal $\langle p \rangle$ is prime. It follows immediately that, in any integral domain, every prime element is irreducible.

\begin{exercise} \label{ex:product of ideals inside a prime ideal}
	Let $P$ be an ideal of a commutative ring $R$ with identity. Show that $P$ is prime if and only if the containment $IJ \subseteq P$ for ideals $I$ and $J$ of $R$ implies that either $I \subseteq P$ or $J \subseteq P$.
\end{exercise}
\medskip

\begin{example} \label{ex:a non-prime ideal}
	We argue that the ideal $I = \langle 2 \rangle$ is not prime in $\fz$ and will in fact use Equation~\eqref{2ndzroot5}. Since $(1 - \sqrt{-5})(1 + \sqrt{-5}) = 2 \cdot 3$, it follows that
	\[
		(1 - \sqrt{-5})(1 + \sqrt{-5}) \in \langle 2 \rangle.
	\]
	Now if $1 - \sqrt{-5} \in \langle 2 \rangle$, then there is an element $\alpha \in \fz$ with $1 - \sqrt{-5} = 2\alpha$. But then $\alpha = \frac{1}{2} - \frac{\sqrt{-5}}{2} \notin \fz$, a contradiction. A similar argument works with $1 + \sqrt{-5}$. Hence, $\langle 2 \rangle$ is not a prime ideal in $\fz$.
\end{example}
\medskip

We remind the reader that a proper ideal $I$ of a commutative ring $R$ with identity is called maximal if for each ideal $J$ the containment $I \subseteq J \subseteq R$ implies that either $J = I$ or $J = R$. What we ask the reader to verify in the next exercise is a well-known result from basic abstract algebra.

\begin{exercise} \label{ex:quotient by prime/maximal ideals are domain/field}
	Let $I$ be a proper ideal of a commutative ring $R$ with identity, and let $R/I = \{r + I \mid r \in R\}$ be the quotient ring of $R$ by $I$.
	\begin{enumerate}
		\item Show that $I$ is prime if and only if $R/I$ is an integral domain. \label{subex:quotient by prime}
		\vspace{3pt}
		\item Show that $I$ is maximal if and only if $R/I$ is a field. Deduce that maximal ideals are prime. \label{subex:quotient by maximal}
	\end{enumerate}
\end{exercise}
\medskip

\begin{example} \label{ex:prime ideal dividing <2>}
	Now we shall expand our analysis of $I = \langle 2, 1+\sqrt{-5}\rangle$ in Example~\ref{ex:a nonprincipal ideal} by showing that $I$ is a prime ideal in $\fz$. To do this, we first argue that an element $\alpha = z_1 + z_2\sqrt{-5} \in \fz$ is contained in $I$ if and only if $z_1$ and $z_2$ have the same parity. If $\alpha \in I$, then there are integers $a, b, c$, and $d$ so that
	\[
		z_1 + z_2\sqrt{-5} = (a+b\sqrt{-5})2 + (c+d\sqrt{-5})(1 + \sqrt{-5}).
	\]
	Adjusting the equations from \eqref{idealarray} yields
	\begin{equation}\label{newidealarray}
		\begin{array}{lclclcl}
			2a & + & c & - & 5d & = & z_1 \\
			2b & + & c & + & d & = & z_2.
		\end{array}
	\end{equation}
	Notice that if $c \equiv d \pmod{2}$, then both $z_1$ and $z_2$ are even, while $c \not\equiv d \pmod{2}$ implies that both $z_1$ and $z_2$ are odd. Hence, $z_1$ and $z_2$ must have the same parity. Conversely, suppose that $z_1$ and $z_2$ have the same parity. As, clearly, every element of the form $2k_1+2k_2\sqrt{-5}=2(k_1+k_2\sqrt{-5})$ is in $I$, let us assume that $z_1$ and $z_2$ are both odd. The equations in \eqref{newidealarray} form a linear system that obviously has solutions over $\q$ for any choice of $z_1$ and $z_2$ in $\z$. By solving this system, we find that $a$ and $b$ are dependent variables and 
	\[
		a = \frac{z_1 -c +5d}{2} \quad \mbox{ and } \quad b=\frac{z_2-c -d}{2}.
	\]
	Letting $c$ be any even integer and $d$ any odd integer now yields a solution with both $a$ and $b$ integers. Thus, $z_1+z_2\sqrt{-5}\in I$.
	
	Now consider $\fz /I$. As $I$ is not principal (Example~\ref{ex:a nonprincipal ideal}), $1 \notin I$. Therefore $1 + I \neq 0 + I$. If $c_1+c_2\sqrt{-5}\not\in I$, then $c_1$ and $c_2$ have opposite parity. If $c_1$ is odd and $c_2$ even, then $((c_1-1) + c_2\sqrt{-5}) + I = 0+I$ implies that $(c_1+c_2\sqrt{-5})+I = 1+ I$. If $c_1$ is even and $c_2$ odd, then $((c_1-1) + c_2\sqrt{-5}) + I = 0+I$ again implies that $(c_1+c_2\sqrt{-5})+I = 1+ I$. Hence, $\fz /I \cong \{0+I, 1+I\} \cong \z_2$. Since $\mathbb{Z}_2$ is a  field, $I$ is a maximal ideal and thus prime (by Exercise~\ref{ex:quotient by prime/maximal ideals are domain/field}).
\end{example}
\medskip

\begin{exercise} \label{ex:prime ideals dividing <3>}
	Show that $\langle 3, 1 - 2 \sqrt{-5} \rangle$ and $\langle 3, 1 + 2 \sqrt{-5} \rangle$ are prime ideals in the ring of integers $\fz$.
\end{exercise}
\medskip

Let $R$ be a commutative ring with identity. If every ideal of $R$ is finitely generated, then $R$ is called a \emph{Noetherian ring}. In addition, $R$ satisfies the \emph{ascending chain condition on ideals} (ACC) if every increasing (under inclusion) sequence of ideals of $R$ eventually stabilizes.

\begin{exercise}
	Let $R$ be a commutative ring with identity. Show that $R$ is Noetherian if and only if it satisfies the ACC.
\end{exercise}
\medskip

We shall see in Theorem~\ref{2-gen} that the rings of integers $\z[\sqrt{d}]$ are Noetherian and, therefore, satisfy the ACC.
\medskip

\section{Ideals in $\fz$} \label{sec:ideals inside the ring of integers}

In this section we explore the algebraic structure of all ideals of $\fz$ under ideal multiplication, encapsulating the basic properties of multiplication of ideals. Let us begin by generalizing the notion of an integral basis, which also plays an important role in ideal theory.

\begin{definition}
	Let $K$ be an algebraic number field of dimension $n$ as a vector space over $\q$, and let $I$ be a proper ideal of the ring of integers $\mathcal{O}_K$. We say that the elements $\omega_1, \dots, \omega_n \in I$ form an \emph{integral basis} for $I$ provided that for each $\beta \in I$ there exist unique $z_1,\dots ,z_n \in \z$ satisfying that $\beta = z_1 \omega_1 + \dots + z_n \omega_n$.
\end{definition}
\medskip

%\begin{definition}
%	Let $I$ be a proper ideal of $\fz$. A set of elements $\{\alpha_1, \alpha_2\}$ of $I$  is called an \textit{integral basis for} $I$ if every element of $I$ can be written uniquely in the form $z_1 \alpha_1 + z_2 \alpha_2$ with $z_1, z_2 \in \z$.
%\end{definition}
%\medskip

With notation as in the above definition, notice that if $\{\omega_1, \dots, \omega_n\}$ is an integral basis for $I$, then $I = \langle \omega_1, \dots, \omega_n \rangle$. Care is needed here as the converse is not necessarily true. For instance, $\{3\}$ is not an integral basis for the ideal $I=\langle 3 \rangle$ of $\fz$ (note  that $3\sqrt{-5} \in I$).

\begin{exercise}
	Argue that $\{3, 3\sqrt{-5}\}$ is an integral basis for the ideal $I = \langle 3 \rangle$ of the ring of integers $\fz$.
\end{exercise}
\medskip

We now show that every proper ideal of $\fz$ has an integral basis.

\begin{theorem} \label{2-gen}
	Let $d \in \z \setminus \{0,1\}$ be a squarefree integer with $d \equiv 2,3 \pmod{4}$. Every nonzero proper ideal of $\z[\sqrt{d}]$ has an integral basis. Hence, every ideal of $\z[\sqrt{d}]$ is finitely generated.
\end{theorem}

\begin{proof}
%	Suppose $I$ is a nonzero proper ideal of $\fz$ with an integral basis $S$. By the definition of an integral basis and an argument similar to that given in the proof of Lemma~\ref{intbasis}, the set $S$ must be linearly independent in $\fq$ as a vector space over $\q$. Hence, $|S| \le 2$. To argue that $I$ cannot have an integral basis consisting of only one element, assume, by contradiction, that $\{\gamma\} \subsetneq \fz^\bullet$ is an integral basis for $I$. Now, take any integral basis $\{\beta_1, \beta_2\}$ for $\fz$ and any $\alpha \in I^\bullet$. As $\alpha \beta_1, \alpha \beta_2 \in I$, they are both integral multiples of $\gamma$ and, therefore, $\{\alpha \beta_1, \alpha \beta_2 \}$ is a linearly dependent set in $\fq$. This implies that $\{\beta_1, \beta_2\}$ is also linearly dependent in $\fz$. But in this case, Lemma~\ref{intbasis} would yield that $\{\beta_1, \beta_2\}$ fails to be an integral basis of $\fz$, a contradiction.
	
	Let $I$ be a nonzero proper ideal of $\z[\sqrt{d}]$. To find an integral basis for $I$ consider the collection $\mathcal{B}$ of all subsets of $I$ which form a vector space basis for $\q(\sqrt{d})$. Note that if $\{\omega_1, \omega_2\}$ is an integral basis for $\z[\sqrt{d}]$ and $\alpha \in I^\bullet$, then the subset $\{\alpha \omega_1, \alpha \omega_2\}$ of $I$ is also a linearly independent subset inside the vector space $\q(\sqrt{d})$. As a result, the collection $\mathcal{B}$ is nonempty. As $I \subseteq \z[\sqrt{d}]$, Proposition~\ref{prop:discrimvalue} ensures that $\Delta[\delta_1,\delta_2] \in \z^\bullet$ for every member $\{\delta_1, \delta_2\}$ of $\mathcal{B}$. Then we can take a pair $\{\delta_1, \delta_2\}$ in $\mathcal{B}$ and assume that the absolute value of its discriminant, i.e., $|\Delta[\delta_1,\delta_2]|$, is as small as possible. We argue now that $\{\delta_1, \delta_2\}$ is an integral basis for $I$.
	
	Assume, by way of contradiction, that $\{\delta_1, \delta_2\}$ is not an integral basis for $I$. Since $\{\delta_1, \delta_2\}$ is a basis for $\q(\sqrt{d})$ as a vector space over $\q$, there must exist $\beta \in I$  and $q_1, q_2 \in \q$ such that $\beta = q_1 \delta_1 + q_2 \delta_2$, where not both $q_1$ and $q_2$ are in $\z$. Without loss of generality, we can assume that $q_1 \in \q \setminus \z$. Write $q_1 = z + r$, where $z \in \z$ and $0 < r < 1$. Let
	\begin{align*}
		\delta_1^\ast &= \beta - z \delta_1 = (q_1 - z) \delta_1 + q_2 \delta_2 \\
		\delta_2^\ast &=\delta_2.
	\end{align*}
	It is easy to verify that $\{\delta_1^\ast, \delta_2^\ast\}$ is linearly independent and thus is another vector space basis for $\q(\sqrt{d})$ which consists of elements of $I$, that is, $\{\delta_1^\ast, \delta_2^\ast\}$ is a member of $\mathcal{B}$. Proceeding as we did in the proof of Proposition~\ref{prop:discrimvalue}, we find that
	\[
		\Delta[\delta_1^\ast, \delta_2^\ast] = r^2\Delta [\delta_1,\delta_2];
	\]
	this is because $\left(\det\left[ \begin{matrix} q_1 - z\;\;\; & q_2 \\ 0 & 1 \end{matrix}\right]\right)^2 = r^2$. It immediately follows from $0<r<~1$ that $|\Delta[\delta_1^\ast, \delta_2^\ast]| < |\Delta[\delta_1,\delta_2]|$, contradicting the minimality of $| \Delta[\delta_1,\delta_2] |$. Hence, $\{\delta_1,\delta_2\}$ is an integral basis for~$I$, which completes the proof.
\end{proof}
\medskip

Theorem~\ref{2-gen} yields the next important corollary.

\begin{corollary} \label{cor:two-elements generating ideals}
	Let $d \in \z \setminus \{0,1\}$ be a squarefree integer with $d \equiv 2,3 \pmod{4}$. If $I$ is a proper ideal of the ring of integers $\z[\sqrt{d}]$, then there exist elements $\alpha_1, \alpha_2 \in I$ such that $I = \langle \alpha_1, \alpha_2 \rangle$. Thus, $\z[\sqrt{d}]$ is a Noetherian ring.
\end{corollary}
\medskip

\begin{remark}
	One can actually say much more. For $d$ as in Corollary~\ref{cor:two-elements generating ideals}, the following stronger statement is true: if $I$ is a nonzero proper ideal of $\z[\sqrt{d}]$ and $\alpha_1 \in I^\bullet$, then there exists $\alpha_2 \in I$ satisfying that $I = \langle \alpha_1, \alpha_2 \rangle$. This condition is known as the $1 \frac{1}{2}$-\emph{generator property}. The interested reader can find a proof of this result in \cite[Theorem~9.3]{PD}.
\end{remark}
\medskip

\begin{definition}
	A pair $(M, *)$, where $M$ is a set and $*$ is a binary operation on $M$, is called a \emph{monoid} if $*$ is associative and there exists $e \in M$ satisfying that $e*x = x*e = x$ for all $x \in M$. The element $e$ is called the \emph{identity element}. The monoid $M$ is called \emph{commutative} if the operation $*$ is commutative.
\end{definition}

Let $R$ be a commutative ring with identity. Recall that we have a natural multiplication on the collection consisting of all ideals of $R$, that is, for any two ideals $I$ and $J$ of $R$, the product
\begin{align} \label{def:ideal multiplication}
	IJ = \big\{\sum_{i=1}^k a_ib_i \ \big{|} \ k \in \n, \ a_1, \dots, a_k \in I, \mbox{ and } b_1, \dots, b_k \in J \big\}
\end{align}
is again an ideal. It is not hard to check that ideal multiplication is both associative and commutative, and satisfies that $R I = I$ for each ideal $I$ of $R$. This amounts to arguing the following exercise.

\begin{exercise} \label{ex:properties of ideal multiplication}
	Let $R$ be a commutative ring with identity. Show that the set of all ideals of $R$ is a commutative monoid under ideal multiplication.
\end{exercise}
\medskip

\begin{example} \label{ex:factorization of <2> into prime ideals}
	To give the reader a notion of how ideal multiplication works, we show that 
	\[
		\langle 2, 1+\sqrt{-5}\rangle^2 = \langle 2\rangle.
	\]
	It follows by~(\ref{def:ideal multiplication}) that ideal multiplication can be achieved by merely multiplying generators. For instance,
	\begin{align*}
		 \langle 2, 1+\sqrt{-5}\rangle^2 &= \langle 2, 1+\sqrt{-5}\rangle \langle 2, 1+\sqrt{-5}\rangle\\
	 	&=\langle 4, 2(1+\sqrt{-5}), 2(1+\sqrt{-5}), -2(2 - \sqrt{-5}) \rangle. %\\
	 	%&= \langle 4, 2(1+\sqrt{-5}),  -4+2\sqrt{-5}\rangle.
	\end{align*}
	Since $2$ divides each of the generators of $\langle 2, 1+\sqrt{-5}\rangle^2$ in $\fz$, we clearly have that $\langle 2, 1+\sqrt{-5}\rangle^2 \subseteq \langle 2\rangle$. To verify the reverse inclusion, let us first observe that
	\[
		2\sqrt{-5} = 4 - 2(2 - \sqrt{-5}) \in \langle 2, 1 + \sqrt{-5} \rangle^2.
	\]
	As $2\sqrt{-5} \in \langle 2, 1 + \sqrt{-5} \rangle^2$, one immediately sees that
	\[
		2 = 2 (1 + \sqrt{-5}) - 2\sqrt{-5} \in \langle 2, 1 + \sqrt{-5} \rangle^2.
	\]
	Hence, the inclusion $\langle 2\rangle\subseteq \langle 2, 1+\sqrt{-5}\rangle^2$ holds, and equality follows.
\end{example}
\medskip

\begin{exercise} \label{ex:factorization of <3> into prime ideals}
	Verify that the next equalities hold:
	\begin{align*}
	\langle 3 \rangle &= \langle 3, 1 - 2\sqrt{-5} \rangle \langle 3, 1 + 2\sqrt{-5} \rangle , \\
	\langle 1 - \sqrt{-5} \rangle &= \langle 2, 1 + \sqrt{-5} \rangle \langle 3, 1 + 2\sqrt{-5} \rangle ,\\% \label{eq:equality of ideals 2} \\
	\langle 1 + \sqrt{-5} \rangle &= \langle 2, 1 + \sqrt{-5} \rangle \langle 3, 1 - 2\sqrt{-5} \rangle . %\label{eq:equality of ideals 3} 
	\end{align*}
\end{exercise}
\medskip

Example~\ref{ex:factorization of <3> into prime ideals} is no accident. Indeed, every nonprincipal ideal of $\fz$ has a multiple which is a principal ideal as it is established in the following theorem.

\begin{theorem} \label{inverse}
	Let $I$ be an ideal of $\fz$. Then there exists a nonzero ideal $J$ of $\fz$ such that $IJ$ is principal.
\end{theorem}

\begin{proof}
	If $I$ is a principal ideal, then the result follows by letting $J=\langle 1\rangle$.  So suppose $I=\langle \alpha, \beta\rangle$ is not a principal ideal of $\fz$, where $\alpha= a+b\sqrt{-5}$ and $\beta=c+d\sqrt{-5}$. Notice that it is enough to verify the existence of such an ideal $J$ when $\gcd(a,b,c,d) = 1$, and we make this assumption. It is easy to check that $\alpha\overline{\beta}+\overline{\alpha}\beta= 2ac+10bd \in \z$. Hence, $\alpha\overline{\alpha}$,  $\alpha \overline{\beta}+\overline{\alpha}\beta$, and $\beta\overline{\beta}$ are all integers. Let
	\begin{align*}
		f &= \gcd(\alpha\overline{\alpha}, \alpha \overline{\beta}+\overline{\alpha}\beta, \beta \bar{\beta}) \\
		  &= \gcd(a^2 + 5b^2, 2ac + 10bd, c^2 + 5d^2).
	\end{align*}
	Take $J = \langle \overline{\alpha}, \overline{\beta}\rangle$. We claim that $IJ = \langle f \rangle$. Since $f = \gcd(\alpha\overline{\alpha}, \alpha \overline{\beta}+\overline{\alpha}\beta, \beta \overline{\beta})$, there are integers $z_1, z_2$, and $z_3$ so that
	\[
		f = z_1  \alpha\overline{\alpha} + z_2 \beta\overline{\beta} + z_3 (\alpha\overline{\beta}+\overline{\alpha}\beta).
	\]
	Because $IJ = \langle \alpha\overline{\alpha},  \alpha\overline{\beta}, \beta \overline{\alpha}, \beta\overline{\beta}\rangle$, we have that $f$ is a linear combination of the generating elements. Thus, $f \in IJ$ and, therefore, $\langle f \rangle \subseteq IJ$.
	
	To prove the reverse containment, we first show that $f$ divides $bc - ad$. Suppose, by way of contradiction, that this is not the case. Notice that $25 \nmid f$; otherwise $25 \mid a^2 + 5b^2$ and $25 \mid c^2 + 5d^2$ would imply that $5 \mid \gcd(a,b,c,d)$. On the other hand, $4 \mid f$ would imply $4 \mid a^2 + 5b^2$ and $4 \mid c^2 + 5d^2$, forcing $a$, $b$, $c$, and $d$ to be even, which is not possible as $\gcd(a,b,c,d) = 1$. Hence, $4 \nmid f$ and $25 \nmid f$. Because
	\begin{align*}
		2c (a^2 + 5b^2) - a(2ac + 10bd) &= 10b(bc - ad) \\
		2a (c^2 + 5d^2) - c(2ac + 10bd) &= 10d(ad - bc),
	\end{align*}
	$f$ must divide both $10b(bc - ad)$ and $10d (bc - ad)$. As, by assumption, $f \nmid bc - ad$, there must be a prime $p$ and a natural $n$ such that $p^n \mid f$ but $p^n \nmid bc - ad$. If $p = 2$, then $4 \nmid f$ forces $n=1$. In this case, both $a^2 + 5b^2$ and $c^2 + 5d^2$ would be even, and so $2 \mid a-b$ and $2 \mid c-d$, which implies that $2 \mid bc - ad$, a contradiction. Thus, $p \neq 2$. On the other hand, if $p=5$, then again $n=1$. In this case, $5 \mid a^2 + 5b^2$ and $5 \mid c^2 + 5d^2$ and so $5$ would divide both $a$ and $c$, contradicting that $5 \nmid bc - ad$. Then, we can assume that $p \notin \{2,5\}$. As $p^n \mid 10b(bc - ad)$ but $p^n \nmid bc - ad$, we have that $p \mid 10b$. Similarly, $p \mid 10d$. Since $p \notin \{2,5\}$, it follows that $p \mid b$ and $p \mid d$. Now the fact that $p$ divides both $a^2 + 5b^2$ and $c^2 + 5d^2$ yields that $p \mid a$ and $p \mid c$, contradicting that $\gcd(a,b,c,d) = 1$. Hence, $f \mid bc - ad$.
	
	Let us verify now that $f \mid ac + 5bd$. If $f$ is odd, then $f \mid ac + 5bd$. Assume, therefore, that $f = 2f_1$, where $f_1 \in \z$. As $4 \nmid f$, the integer $f_1$ is odd. Now, $f \mid a^2 + 5b^2$ implies that $a$ and $b$ have the same parity. Similarly, one sees that $c$ and $d$ have the same parity. As a consequence, $ac + 5bd$ is even. Since $f_1$ is odd, it must divide $(ac + 5bd)/2$, which means that $f$ divides $ac + 5bd$, as desired.
	
	Because $f$ divides both $\alpha \overline{\alpha}$ and $\beta \overline{\beta}$ in $\z$, proving that $IJ \subseteq \langle f \rangle$ amounts to verifying that $f$ divides both $\alpha \overline{\beta}$ and $\overline{\alpha} \beta$ in $\fz$. Since $f$ divides both $ac + 5 bd$ and $bc - ad$ in $\z$, one has that
	\[
		x = \frac{ac + 5bd}{f} \in \z \quad \text{ and } \quad y = \frac{bc - ad}{f} \in \z.
	\]
	Therefore
	\begin{align*}
		\alpha \overline{\beta} &= ac + 5bd + (bc - ad)\sqrt{-5} = (x + y \sqrt{-5})f \in \langle f \rangle.
	\end{align*}
	Also, $\overline{\alpha} \beta = \overline{\alpha \overline{\beta}} = (x - y \sqrt{-5})f \in \langle f \rangle$. Hence, the reverse inclusion $IJ \subseteq \langle f \rangle$ also holds, which completes the proof.
\end{proof}
\medskip

A commutative monoid $(M,*)$ is said to be \emph{cancellative} if for all $a,b,c \in M$, the equality $a*b = a*c$ implies that $b = c$. By Exercise~\ref{ex:properties of ideal multiplication}, the set
\[
	\cI := \{I \mid I \ \text{is an ideal of } \ \fz\}
\]
is a commutative monoid. As the next corollary states, the set $\cI^\bullet := \cI \setminus \{\langle 0 \rangle \}$  is indeed a commutative cancellative monoid.

\begin{corollary}
	The set $\cI^\bullet$ under ideal multiplication is a commutative cancellative monoid.
\end{corollary}

\begin{proof}
	Because $\cI$ is a commutative monoid under ideal multiplication, it immediately follows that $\cI^\bullet$ is also a commutative monoid. To prove that $\cI^\bullet$ is cancellative, take $I,J,K \in \cI^\bullet$ such that $IJ = IK$. By Theorem~\ref{inverse}, there exists an ideal $I'$ of $\fz$ and $x \in \fz^\bullet$ with $I'I = \langle x \rangle$. Then
	\[
		\langle x \rangle J = I'IJ = I'IK = \langle x \rangle K.
	\]
	As $x \neq 0$ and the product in $\fz^\bullet$ is cancellative, $J = K$.% This completes the proof
\end{proof}
\medskip

\section{The Fundamental Theorem of Ideal Theory} \label{sec:fundamental theorem}

We devote this section to prove a version of the Fundamental Theorem of Ideal Theory for the ring of integers $\fz$. To do this, we need to develop a few tools. In particular, we introduce the concept of a fractional ideal of $\fz$ and show that the set of such fractional ideals is an abelian group.

Let us begin by exploring the relationship between the concepts of prime and maximal ideals. We recall that every proper ideal of a commutative ring $R$ with identity is contained in a maximal ideal, which implies, in particular, that maximal ideals always exist.

\begin{exercise} \label{ex:maximal ideals are prime}
	Show that every maximal ideal of a commutative ring with identity is prime.
\end{exercise}
\medskip

Prime ideals, however, are not necessarily maximal. The following example sheds some light upon this observation.

\begin{example}
	Let $\z[X]$ denote the ring of polynomials with integer coefficients. Clearly, $\z[X]$ is an integral domain. It is not hard to verify that the ideal $\langle X \rangle$ of $\z[X]$ is prime. Because $2 \notin \langle X \rangle$, one obtains that $\langle X \rangle \subsetneq \langle 2, X \rangle$. It is left to the reader to argue that $\langle 2, X \rangle$ is a proper ideal of $\z[X]$. Since $\langle X \rangle \subsetneq \langle 2, X \rangle \subsetneq \z[X]$, it follows that $\langle X \rangle$ is not a maximal ideal of $\z[X]$.
	(An alternate argument can easily be given using Exercise~\ref{ex:quotient by prime/maximal ideals are domain/field}.)
\end{example}
\medskip

In the ring of integers $\mathcal{O}_K$ of any algebraic number field $K$, every nonzero prime ideal is maximal (see, for instance, \cite[Proposition~5.21]{Ja}). Let us establish this result here for our case of interest.

\begin{proposition} \label{prop:prime ideals are maximal}
	Let $d \in \z \setminus \{0,1\}$ be a squarefree integer with $d \equiv 2,3 \pmod{4}$. Then every nonzero prime ideal of $\z[\sqrt{d}]$ is maximal.
\end{proposition}

\begin{proof}
	Let $P$ be a nonzero prime ideal in $\z[\sqrt{d}]$, and let $\{\omega_1, \omega_2\}$ be an integral basis for $\z[\sqrt{d}]$. Fix $\beta \in P^\bullet$. Note that $n := N(\beta) = \beta \bar{\beta} \in P \cap \n$. Consider the finite subset
	\[
		S = \big\{ n_1 \omega_1 + n_2 \omega_2 + P \mid n_1, n_2 \in \{0,1, \dots, n-1\} \big\}
	\]
	of $\z[\sqrt{d}]/P$. Take $x \in \z[\sqrt{d}]$. As $\{\omega_1, \omega_2\}$ is an integral basis, there exist $z_1, z_2 \in \z$ such that $x = z_1 \omega_1 + z_2 \omega_2$ and, therefore, $x + P = n_1 \omega_1 + n_2 \omega_2 + P \in S$, where $n_i \in \{0, \dots, n-1\}$ and $n_i \equiv z_i \pmod{n}$. Hence, $\z[\sqrt{d}]/P = S$, which implies that $\z[\sqrt{d}]/P$ is finite. It follows by Exercise~\ref{ex:quotient by prime/maximal ideals are domain/field}(\ref{subex:quotient by prime}) that $\z[\sqrt{d}]/P$ is an integral domain. As a result, $\z[\sqrt{d}]/P$ is a field (see Exercise~\ref{ex:finite domains are fields} below). Thus, Exercise~\ref{ex:quotient by prime/maximal ideals are domain/field}(\ref{subex:quotient by maximal}) guarantees that $P$ is a maximal ideal.
\end{proof}
\medskip

\begin{exercise} \label{ex:finite domains are fields}
	Let $R$ be a finite integral domain. Show that $R$ is a field.
\end{exercise}
\medskip

Although the concepts of (nonzero) prime and maximal ideals coincide in $\z[\sqrt{d}]$, we will use both terms depending on the ideal property we are willing to apply.

\begin{lemma} \label{lem:some product of prime ideals is contained in a given ideal}
	If $I$ is a nonzero ideal of a Noetherian ring $R$, then there exist nonzero prime ideals $P_1, \dots, P_n$ of $R$ such that $P_1 \cdots P_n \subseteq I$.
\end{lemma}

\begin{proof}
	Assume, by way of contradiction, that the statement of the lemma does not hold. Because $R$ is a Noetherian ring and, therefore, satisfies the ACC, there exists an ideal $I$ of $R$ that is maximal among all the ideals failing to satisfy the statement of the lemma. Clearly, $I$ cannot be prime. By Exercise~\ref{ex:product of ideals inside a prime ideal}, there exist ideals $J$ and $K$ of $R$ such that $JK \subseteq I$ but neither $J \subseteq I$ nor $K \subseteq I$. Now notice that the ideals $J' = I + J$ and $K' = I + K$ both strictly contain $I$. The maximality of $I$ implies that both $J'$ and $K'$ contain products of nonzero prime ideals. Now the fact that $J'K' \subseteq I$ would also imply that $I$ contains a product of nonzero prime ideals, a contradiction.
\end{proof}
\medskip

Recall that if $R$ is an integral domain contained in a field $F$, then the field of fractions of $R$ is the smallest subfield of $F$ containing $R$. If $K$ is an algebraic number field, then it is not hard to argue that the field of fractions of $\mathcal{O}_K$ is precisely $K$.

\begin{definition}
	Let $R$ be an integral domain with field of fractions $F$. A \emph{fractional ideal} of $R$ is a subset of $F$ of the form $\alpha^{-1}I$, where $\alpha \in R^\bullet$ and $I$ is an ideal of $R$. 
\end{definition}
\medskip

%\begin{definition}
%	For $I \in \cI$ and $\alpha \in \fz^\bullet$, the subset $\alpha^{-1}I$ of $\fq$ is called a \emph{fractional ideal} of $\fz$. Let $\mathcal{F}$ denote the set of all fractional ideals of $\fz$.
%\end{definition}
%\medskip

With notation as in the previous definition, it is clear that every ideal of $R$ is a fractional ideal. However, fractional ideals are not necessarily ideals. The product of fractional ideals is defined similarly to the product of standard ideals. Therefore it is easily seen that the product of two fractional ideals is again a fractional ideal. Indeed, for elements $\alpha$ and $\beta$ of $R^\bullet$ and for ideals $I$ and $J$ of $R$, we only need to observe that $(\alpha^{-1}I)(\beta^{-1}J) = (\alpha \beta)^{-1} IJ$.
\medskip

\noindent {\bf Notation:} Let $\mathcal{F}$ denote the set of all fractional ideals of $\fz$, and let $\mathcal{F}^\bullet$ denote the set $\mathcal{F} \setminus \{\langle 0 \rangle \}$.
\medskip

\begin{definition}
	Let $R$ be an integral domain with field of fractions $F$. For a fractional ideal $I$ of $R$, the set
	\[
		I^{-1} := \{\alpha \in F \mid \alpha I \subseteq R\}
	\]
	is called the \emph{inverse} of $I$.
\end{definition}
\medskip

\begin{exercise} \label{ex:inverse ideals are fractional}
	Show that the inverse of a fractional ideal is again a fractional ideal.
\end{exercise}
\medskip

\begin{lemma} \label{lem:inverse of proper ideals strictly contain the ring}
	Let $d \in \z \setminus \{0,1\}$ be a squarefree integer with $d \equiv 2,3 \pmod{4}$. If $I$ is a proper ideal of the ring of integers $\z[\sqrt{d}]$, then $\z[\sqrt{d}]$ is strictly contained in the fractional ideal $I^{-1}$.
\end{lemma}

\begin{proof}
	Since $I$ is a proper ideal of $\z[\sqrt{d}]$, then there exists a maximal ideal $M$ of $\z[\sqrt{d}]$ containing $I$. Fix $\alpha \in M^\bullet$. By the definition of the inverse of an ideal, $\z[\sqrt{d}] \subseteq M^{-1}$. Since $\z[\sqrt{d}]$ is a Noetherian ring, Lemma~\ref{lem:some product of prime ideals is contained in a given ideal} ensures the existence of $m \in \n$ and prime ideals $P_1, \dots, P_m$ in $\z[\sqrt{d}]$ such that $P_1 \cdots P_m \subseteq \langle \alpha \rangle \subseteq M$. Assume that $m$ is the minimum natural number satisfying this property. Since $M$ is a prime ideal (Exercise~\ref{ex:maximal ideals are prime}), by Exercise~\ref{ex:product of ideals inside a prime ideal} there exists $P \in \{P_1, \dots, P_m\}$ such that $P \subseteq M$. There is no loss of generality in assuming that $P = P_1$. Now it follows by Proposition~\ref{prop:prime ideals are maximal}, the ideal $P_1$ is maximal, which implies that $P_1 = M$. By the minimality of $m$, there exists $\alpha' \in P_2 \cdots P_m \setminus \langle \alpha \rangle$. Thus, $\alpha^{-1} \alpha' \notin \z[\sqrt{d}]$ and $\alpha' M = \alpha' P_1 \subseteq P_1 \cdots P_m \subseteq \langle \alpha \rangle$, that is $\alpha^{-1}\alpha' M \subseteq \langle 1 \rangle = \z[\sqrt{d}]$. As a result, $\alpha^{-1} \alpha' \in M^{-1} \setminus \z[\sqrt{d}]$. Hence, we find that $\z[\sqrt{d}] \subsetneq M^{-1} \subseteq I^{-1}$, and the proof follows.
\end{proof}
\medskip

We focus throughout the remainder of our work on the ring of integers $\fz$.  This, via Theorem~\ref{inverse}, will substantially simplify our remaining arguments.

\begin{lemma} \label{lem:multiplying ideal by element of Q(sqrt{-5})}
	If $I \in \cI^\bullet$ and $\alpha \in \fq$, then $\alpha I \subseteq I$ implies $\alpha \in \fz$.
\end{lemma}

\begin{proof}
	Let $I$ and $\alpha$ be as in the statement of the lemma. By Theorem~\ref{inverse}, there exists a nonzero ideal $J$ of $\fz$ such that $IJ = \langle \beta \rangle$ for some $\beta \in \fz$. Then $\alpha \langle \beta \rangle = \alpha IJ \subseteq IJ = \langle \beta \rangle$, which means that $\alpha \beta = \sigma \beta$ for some $\sigma \in \fz$. As $\beta \neq 0$, it follows that $\alpha = \sigma \in \fz$.
\end{proof}
\medskip

\begin{theorem} \label{thm:the set of fractional ideals forms a group}
	The set $\mathcal{F}^\bullet$ is an abelian group under multiplication of fractional ideals.
\end{theorem}

\begin{proof}
	Clearly, multiplication of fractional ideals is associative. In addition, it immediately follows that the fractional ideal $\fz = 1^{-1} \langle 1 \rangle$ is the identity. The most involved part of the proof consists in arguing that each fractional ideal is invertible.
	
	Let $M \in \cI^\bullet$ be a maximal ideal of $\fz$. By definition of $M^{-1}$, we have that $M M^{-1} \subseteq \fz$, which implies that $M M^{-1} \in \cI^\bullet$. As $M = M \fz \subseteq M M^{-1}$ and $M$ is maximal, $M M^{-1} = M$ or $M M^{-1} = \fz$. As $M$ is proper, Lemma~\ref{lem:inverse of proper ideals strictly contain the ring} ensures that $M^{-1}$ strictly contains $\fz$, which implies, by Lemma~\ref{lem:multiplying ideal by element of Q(sqrt{-5})}, that $M M^{-1} \neq M$. So $M M^{-1} = \fz$. As a result, each maximal ideal of $\fz$ is invertible.
	
	Now suppose, by way of contradiction, that not every ideal in $\mathcal{I}^\bullet$ is invertible. Among all the nonzero non-invertible ideals take one, say $J$, maximal under inclusion (this is possible because $\fz$ satisfies the ACC). Because $\fz$ is an invertible fractional ideal, $J \subsetneq \fz$. Let $M$ be a maximal ideal containing $J$. By Lemma~\ref{lem:inverse of proper ideals strictly contain the ring}, one has that $\fz \subsetneq M^{-1} \subseteq J^{-1}$. This, along with Lemma~\ref{lem:multiplying ideal by element of Q(sqrt{-5})}, yields $J \subsetneq J M^{-1} \subseteq J J^{-1} \subseteq \fz$. Thus, $J M^{-1}$ is an ideal of $\fz$ strictly containing $J$. The maximality of $J$ now implies that $J M^{-1}(J M^{-1})^{-1} = \fz$ and, therefore, $M^{-1}(J M^{-1})^{-1} \subseteq J^{-1}$. Then
	\[
		\fz = J M^{-1}(J M^{-1})^{-1} \subseteq J J^{-1} \subseteq \fz,
	\]
	which forces $J J^{-1} = \fz$, a contradiction.
	
	Finally, take $F \in \mathcal{F}^\bullet$. Then there exist an ideal $I \in \mathcal{I}^\bullet$ and $\alpha \in \fz^\bullet$ such that $F = \alpha^{-1} I$. So one obtains that
	\[
		(\alpha I^{-1}) F = (\alpha I^{-1})(\alpha^{-1} I) = I^{-1} I = \fz.
	\]
	As a consequence, the fractional ideal $\alpha I^{-1}$ is the inverse of $F$ in $\mathcal{F}^\bullet$. Because each nonzero fractional ideal of $\fz$ is invertible, $\mathcal{F}^\bullet$ is a group. Since the multiplication of fractional ideals is commutative, $\mathcal{F}^\bullet$ is abelian.
\end{proof}
\medskip

\begin{corollary} \label{cor:an ideal divides any principal ideal it contains}
	If $I \in \cI^\bullet$ and $\alpha \in I^\bullet$, then $IJ = \langle \alpha \rangle$ for some $J \in \cI^\bullet$.
\end{corollary}

\begin{proof}
	Let $I$ and $\alpha$ be as in the statement of the corollary. As $\alpha^{-1} I$ is a nonzero fractional ideal, there exists a nonzero fractional ideal $J$ such that $\alpha^{-1}I J = \fz$, that is $I J = \langle \alpha \rangle$. Since $\beta I \subseteq JI = \langle \alpha \rangle \subseteq I$ for all $\beta \in J$, Lemma~\ref{lem:multiplying ideal by element of Q(sqrt{-5})} guarantees that $J \subseteq \fz$. Hence, $J$ is a nonzero ideal of $\fz$.
\end{proof}
\medskip

\begin{theorem}\label{FTIT} [The Fundamental Theorem of Ideal Theory]
	Let $I$ be a nonzero proper ideal of $\fz$. There exists a unique (up to order)
	list of prime ideals $P_1, \dots ,P_k$ of $\fz$ such that $I = P_1 \cdots P_k$.
\end{theorem}

\begin{proof}
	Suppose, by way of contradiction, that not every ideal in $\mathcal{I}^\bullet$ can be written as the product of prime ideals. From the set of ideals of $\fz$ which are not the product of primes ideals, take one, say $I$, maximal under inclusion. Clearly, $I$ is not prime. Therefore $I$ is contained in a maximal ideal $P_1$, and such containment must be strict by Exercise~\ref{ex:maximal ideals are prime}. By Lemma~\ref{lem:inverse of proper ideals strictly contain the ring}, one has that $\fz \subsetneq P_1^{-1}$ and so $I \subseteq I P_1^{-1}$. Now Lemma~\ref{lem:multiplying ideal by element of Q(sqrt{-5})} ensures that the latter inclusion is strict. The maximality of $I$ now implies that $I P_1^{-1} = P_2 \cdots P_k$ for some prime ideals $P_2, \dots, P_k$. This, along with Theorem~\ref{thm:the set of fractional ideals forms a group}, ensures that $I = P_1 \cdots P_k$, a contradiction.
	
	To argue uniqueness, let us assume, by contradiction, that there exists an ideal having two distinct prime factorizations. Let $m$ be the minimum natural number such that there exists $I \in \cI$ with two distinct factorizations into prime ideals, one of them containing $m$ factors. Suppose that
	\begin{align} \label{eq:distinct prime factorizations}
		I = P_1 \cdots P_m = Q_1 \cdots Q_n.
	\end{align}
	Because $Q_1 \cdots Q_n \subseteq P_m$, there exists $Q \in \{Q_1, \dots, Q_n\}$ such that $Q \subseteq P_m$ (Exercise~\ref{ex:product of ideals inside a prime ideal}). By Proposition~\ref{prop:prime ideals are maximal}, both $Q$ and $P_m$ are maximal ideals, which implies that $P_m = Q$. As $IQ^{-1} \subseteq I I^{-1} \subseteq \fz$, it follows that $I Q^{-1} \in \mathcal{I}$. Multiplying the equality (\ref{eq:distinct prime factorizations}) by the fractional ideal $Q^{-1}$, we obtain that $I Q^{-1}$ is an ideal of $\fz$ with two distinct factorizations into prime ideals such that one of them, namely $P_1 \cdots P_{m-1}$, contains less than $m$ factors. As this contradicts the minimality of $m$, uniqueness follows.
\end{proof}
\medskip

An element $a$ of a commutative monoid $M$ is said to be an \emph{atom} if for all $x,y \in M$ such that $a = xy$, either $x$ is a unit or $y$ is a unit (i.e., has an inverse). A commutative cancellative monoid is called \emph{atomic} if every nonzero nonunit element can be factored into atoms.

\begin{corollary}
	The monoid $\cI^\bullet$ is atomic.
\end{corollary}
\medskip

\section{The Class Group} \label{sec:the class group}

To understand the phenomenon of non-unique factorization in $\fz$, we first need to understand certain classes of ideals of $\fz$. Let
\[
	\mathcal{P} := \{ I \in \mathcal{I} \mid I \mbox{ is a principal ideal of } \fz\}.
\]
Two ideals $I,J \in \cI$ are \emph{equivalent} if $\langle \alpha \rangle I = \langle \beta \rangle J$ for some $\alpha, \beta \in \fz^\bullet$. In this case, we write $I \sim J$. It is clear that $\sim$ defines an equivalence relation on $\fz$. The equivalence classes of $\sim$ are called \emph{ideal classes}. Let $I\mathcal{P}$ denote the ideal class of $I$, and we also let $\mathcal{C}(\fz)$ denote the set of all nonzero ideal classes. Now define a binary operation $\ast$ on $\mathcal{C}(\fz)$ by
\[
	I \mathcal{P} \ast J\mathcal{P} = (IJ)\mathcal{P}.
\]
It turns out that $\mathcal{C}(\fz)$ is, indeed, a group under the $\ast$ operation.

\begin{theorem}
	The set of ideal classes $\mathcal{C}(\fz)$ is an abelian group under $\ast$.
\end{theorem}

\begin{proof}
	Because the product of ideals is associative and commutative, so is $\ast$. Also, it follows immediately that $\langle 1 \rangle \mathcal P \ast I \mathcal P =  (\langle 1 \rangle I) \mathcal P = I \mathcal P$ for each $I \in \cI^\bullet$, which means that $\mathcal P = \langle 1 \rangle \mathcal{P}$ is the identity element of $\mathcal{C}(\fz)$. In addition, as any two nonzero principal ideals are in the same ideal class, Theorem~\ref{inverse} guarantees that, for any $I \mathcal P \in \mathcal C (\fz)$, there exists $J \in \cI^\bullet$ such that $I\mathcal{P} \ast J\mathcal{P} = IJ \in \mathcal P = \langle 1 \rangle \mathcal{P}$. So $J \mathcal P$ is the inverse of $I \mathcal P$ in $\mathcal{C}(\fz)$. Hence, $\mathcal C (\fz)$ is an abelian group.
\end{proof}
\medskip

\begin{definition}
	The group $\mathcal{C}(\fz)$ is called the \emph{class group} of $\fz$, and the order of $\mathcal{C}(\fz)$ is called the \emph{class number} of $\fz$.
\end{definition}
\medskip

Recall that if $\theta \colon R \to S$ is a ring homomorphism, then $\ker \theta = \{r \in R \mid \theta(r) = 0\}$ is an ideal of $R$. Moreover, the First Isomorphism Theorem for rings states that $R/\ker \theta \cong \theta(R)$.

\begin{definition}\label{idealnorm}
	Let $K$ be an algebraic number field. For any nonzero ideal $I$ of $\mathcal{O}_K$, the cardinality $|\mathcal{O}_K/I|$ is called the \emph{norm} of $I$ and is denoted by $N(I)$.
%	If $I \in \cI$, then the \emph{norm} of $I$, denoted by $N(I)$, is the size of the quotient ring $\fz/I$.
\end{definition}
\medskip

%Definition~\ref{idealnorm} is specific to $\fz$, but generalizes in a natural way to any ring of integers.

\begin{proposition} \label{prop:there are only finitely many residue classes}
	Let $d \in \z \setminus \{0,1\}$ be a squarefree integer with $d \equiv 2,3 \pmod{4}$. Then $N(I)$ is finite for all nonzero ideals $I$ of $\z[\sqrt{d}]$.
\end{proposition}

\begin{proof}
	Take $n = \alpha \bar{\alpha}$ for any nonzero $\alpha \in I$. Then $n \in I \cap \n$. As $\langle n \rangle \subseteq I$, it follows that $|\z[\sqrt{d}]/I| \le |\z[\sqrt{d}]/\langle n \rangle|$. In addition, each element of $\z[\sqrt{d}]/\langle n \rangle$ has a representative $n_1 + n_2\sqrt{d}$ with $n_1, n_2 \in \{0,1,\dots, n-1\}$. Hence, $\z[\sqrt{d}]/\langle n \rangle$ is finite and, therefore, $N(I) = |\z[\sqrt{d}]/I| < \infty$.
\end{proof}
\medskip

As ideal norms generalize the notion of standard norms given in (\ref{eq:standard norm}), we expect they satisfy some similar properties. Indeed, this is the case.

\begin{exercise} \label{ex:cardinality of quotients}
	Let $I$ and $P$ be a nonzero ideal and a nonzero prime ideal of $\fz$, respectively. Show that $|\fz/P| = |I/IP|$.
\end{exercise}
\medskip

\begin{proposition} \label{prop:norm is multiplicative}
	$N(IJ) = N(I)N(J)$ for all $I,J \in \cI^\bullet$.
\end{proposition}

\begin{proof}
	By factoring $J$ as the product of prime ideals (Theorem~\ref{FTIT}) and applying induction on the number of factors, we can assume that $J$ is a prime ideal. Consider the ring homomorphism $\theta \colon \fz/IJ \to \fz/I$ defined by $\theta(\alpha + IJ) = \alpha + I$. It follows immediately that $\theta$ is surjective and $\ker \theta = \{\alpha + IJ \mid \alpha \in I\}$. Therefore
	\[
		\frac{\fz/IJ}{I/IJ} \cong \fz/I
	\]
	by the First Isomorphism Theorem. As $IJ$ is nonzero, $|\fz/IJ| = N(IJ)$ is finite and so $|\fz/IJ| = |\fz/I| \cdot |I/IJ|$. Since $J$ is prime, we can use Exercise~\ref{ex:cardinality of quotients} to conclude that
	\begin{align*}
		N(IJ) &= |\fz/IJ| = |\fz/I| \cdot |I/IJ|\\
				&= |\fz/I| \cdot |\fz/J| = N(I)N(J).
	\end{align*}
\end{proof}
\medskip

\begin{corollary} \label{cor:prime norm implies prime ideal}
	If $N(I)$ is prime for some $I \in \cI^\bullet$, then $I$ is a prime ideal.
\end{corollary}
\medskip

Let us verify now that the ideal norm is consistent with the standard norm on principal ideals.

\begin{proposition} \label{prop:norm ideal of a principal ideal}
	$N(\langle \alpha \rangle) = N(\alpha)$ for all $\alpha \in \fz^\bullet$.
\end{proposition}

\begin{proof}
	Set $S = \big\{a + b\sqrt{-5} \mid a,b \in \{0,1, \dots, n-1\} \big\}$. Clearly, $|S| = n^2$. In addition,
	\[
		\fz/ \langle n \rangle = \{s + \langle n \rangle \mid s \in S\}.
	\]
	Note that if $s + \langle n \rangle = s' + \langle n \rangle$ for $s,s' \in S$, then we have $s = s'$. As a consequence, $N(\langle n \rangle) = n^2 = N(n)$ for each $n \in \n$. It can also be readily verified that the map $\theta \colon \fz \to \fz/ \langle \bar{\alpha} \rangle$ defined by $\theta(x) = \bar{x} + \langle \bar{\alpha} \rangle$ is a surjective ring homomorphism with $\ker \theta = \langle \alpha \rangle$. Therefore the rings $\fz/ \langle \alpha \rangle$ and $\fz/ \langle \bar{\alpha} \rangle$ are isomorphic by the First Isomorphism Theorem. This implies that $N(\langle \alpha \rangle) = N(\langle \bar{\alpha} \rangle)$. Because $\alpha \bar{\alpha} \in \n$, using Proposition~\ref{prop:norm is multiplicative}, one obtains
	\[
		N(\langle \alpha \rangle) = \sqrt{N(\langle \alpha \rangle)N(\langle\bar{\alpha} \rangle)} = \sqrt{N(\langle \alpha \bar{\alpha} \rangle)} = \alpha \bar{\alpha} = N(\alpha).
	\]
\end{proof}
\medskip

\begin{lemma} \label{lem:prime ideal as a factor of principal integer ideal}
	If $P$ is a nonzero prime ideal in $\fz$, then $P$ divides exactly one ideal $\langle p \rangle$, where $p$ is a prime number.
\end{lemma}

\begin{proof}
	For $\alpha \in P^\bullet$, it follows that $z = \alpha \bar{\alpha} \in P \cap \n$. Then, writing $z = p_1 \cdots p_k$ for some prime numbers $p_1, \dots, p_k$, we get $\langle z \rangle = \langle p_1 \rangle \cdots \langle p_k \rangle$. As $\langle p_1 \rangle \cdots \langle p_k \rangle \subseteq P$, we have that $\langle p_i \rangle \subseteq P$ for some $i \in \{1, \dots, k\}$ (Exercise~\ref{ex:product of ideals inside a prime ideal}). As $p_i \in P^\bullet$, Corollary~\ref{cor:an ideal divides any principal ideal it contains} ensures that $P$ divides $\langle p_i \rangle$. For the uniqueness, note that if $P$ divides $\langle p \rangle$ and $\langle p' \rangle$ for distinct primes $p$ and $p'$, then the fact that $mp + np' = 1$ for some $m,n \in \z$ would imply that $P$ divides the full ideal $\langle 1 \rangle = \fz$, a contradiction. \\
\end{proof}
\medskip

\begin{theorem}
	The class group of $\fz$ is $\z_2$.
\end{theorem}

\begin{proof}
	First, we verify that every nonzero ideal $I$ of $\fz$ contains a nonzero element $\alpha$ with $N(\alpha) \le 6 N(I)$. For $I \in \mathcal{I}^\bullet$, take $B = \lfloor \sqrt{N(I)} \rfloor$ and define
	\[
		S_I := \big\{ a + b \sqrt{-5} \mid a,b \in \{0,1, \dots, B\} \big\} \subsetneq \fz.
	\]
	Observe that $|S_I| = (B+1)^2 > N(I)$. Thus, there exist $\alpha_1 = a_1 + b_1 \sqrt{-5} \in S_I$ and $\alpha_2 = a_2 + b_2 \sqrt{-5} \in S_I$ such that $\alpha = \alpha_1 - \alpha_2 \in I \setminus \{0\}$ and
	\[
		N(\alpha) = (a_1 - a_2)^2 + 5(b_1 - b_2)^2 \le 6B^2 \le 6 N(I).
	\]
	
	Now, let $I \mathcal{P}$ be a nonzero ideal class of $\fz$. Take $J \in \cI^\bullet$ satisfying $IJ \mathcal{P} = \mathcal{P}$. By the argument given in the previous paragraph, there exists $\beta \in J^\bullet$ such that $N(\beta) \le 6N(J)$. By Corollary~\ref{cor:an ideal divides any principal ideal it contains}, there exists an ideal $K \in \cI^\bullet$ such that $JK = \langle \beta \rangle$. Using Proposition~\ref{prop:norm is multiplicative} and Proposition~\ref{prop:norm ideal of a principal ideal}, one obtains
	\[
		N(J)N(K) = N(\langle \beta \rangle) = N(\beta) \le 6 N(J),
	\]
	which implies that $N(K) \le 6$. Because $KJ \sim IJ$ (they are both principal), it follows that $K \in I\mathcal{P}$. Hence, every nonzero ideal class of $\fz$ contains an ideal whose norm is at most $6$.
	
	To show that the class group of $\fz$ is $\z_2$, let us first determine the congruence relations among ideals of norm at most $6$. Every ideal $P$ of norm $p \in \{2,3,5\}$ must be prime by Corollary~\ref{cor:prime norm implies prime ideal}. Moreover, by Lemma~\ref{lem:prime ideal as a factor of principal integer ideal}, Theorem~\ref{FTIT}, and Proposition~\ref{prop:norm is multiplicative}, the ideal $P$ must show in the prime factorization
	\begin{align} \label{factorization of (p)}
		\langle p \rangle = P_1^{n_1} \cdots P_k^{n_k}
	\end{align}
	of the ideal $\langle p \rangle$. The following ideal factorizations have been already verified in Example~\ref{ex:factorization of <2> into prime ideals} and Exercise~\ref{ex:factorization of <3> into prime ideals}:
	\begin{align} \label{eq:equality of ideals 1}
		\langle 2 \rangle &= \langle 2, 1 + \sqrt{-5} \rangle^2, \nonumber \\
		\langle 3 \rangle &= \langle 3, 1 - 2\sqrt{-5} \rangle \langle 3, 1 + 2\sqrt{-5} \rangle, \\
		\langle 5 \rangle &= \langle \sqrt{-5} \rangle^2. \nonumber
	\end{align}
	%for $p \in \{2,3,5\}$.
	In addition, we have proved in Example~\ref{ex:prime ideal dividing <2>} and Exercise~\ref{ex:prime ideals dividing <3>} that the ideals on the right-hand side of the first two equalities in~(\ref{eq:equality of ideals 1}) are prime. Also, $N(\langle \sqrt{-5} \rangle) = N(\sqrt{-5}) = 5$ implies that the ideal $\langle \sqrt{-5} \rangle$ is prime. It follows now by the uniqueness of Theorem~\ref{FTIT} that the ideals on the right-hand side of the equalities~(\ref{eq:equality of ideals 1}) are the only ideals of $\fz$ having norm in the set $\{2,3,5\}$. Once again, combining Lemma~\ref{lem:prime ideal as a factor of principal integer ideal}, Theorem~\ref{FTIT}, and Proposition~\ref{prop:norm is multiplicative}, we obtain that any ideal $I$ whose norm is $4$ must be a product of prime ideals dividing $\langle 2 \rangle$, which forces $I = \langle 2 \rangle$. Similarly, any ideal $J$ with norm $6$ must be the product of ideals dividing the ideals $\langle 2 \rangle$ and $\langle 3 \rangle$. The reader can readily verify that,
	\begin{align}
		\langle 1 - \sqrt{-5} \rangle &= \langle 2, 1 + \sqrt{-5} \rangle \langle 3, 1 + 2\sqrt{-5} \rangle \label{eq:equality of ideals 2} \\
		\langle 1 + \sqrt{-5} \rangle &= \langle 2, 1 + \sqrt{-5} \rangle \langle 3, 1 - 2\sqrt{-5} \rangle. \label{eq:equality of ideals 3} 
	\end{align}
	Therefore $\langle 1 - \sqrt{-5} \rangle$ and $\langle 1 + \sqrt{-5} \rangle$ are the only two ideals having norm $6$. Now since we know all ideals of $\fz$ with norm at most 6, it is not difficult to check that $|\mathcal{C}(\fz)| = 2$. Because each principal ideal of $\fz$ represents the identity ideal class $\mathcal{P}$, we find that
	\[
		\langle 1 \rangle \mathcal{P} = \langle 2 \rangle \mathcal{P} = \langle \sqrt{-5} \rangle \mathcal{P} = \langle 1 - \sqrt{-5} \rangle \mathcal{P}.
	\]
	On the other hand, we have seen that the product of $\langle 2, 1 + \sqrt{-5} \rangle$ and each of the three nonprincipal ideals with norm at most $6$ is a principal ideal. Thus,
	\[
		\langle 2, 1 + \sqrt{-5} \rangle \mathcal{P} = \langle 3, 1 + 2 \sqrt{-5}\rangle \mathcal{P} = \langle 3, 1 - 2\sqrt{-5} \rangle \mathcal{P}.
	\]
	Since there are only two ideal classes, $\mathcal{C}(\fz) = \z_2$.
\end{proof}
\medskip

\begin{exercise}
	Verify the equalities (\ref{eq:equality of ideals 2}), and (\ref{eq:equality of ideals 3}).
\end{exercise}
\medskip

From this observation, we deduce an important property of the ideals of $\fz$.

\begin{corollary}\label{multcorollary}
	If $I, J \in \cI^\bullet$ are not principal, then $IJ$ is principal.
\end{corollary}
\medskip

\section{Half-factoriality} \label{sec:half-factoriality}

The class group, in tandem with The Fundamental Theorem of Ideal Theory, will allow us to determine exactly what elements of $\fz$ are irreducible.

\begin{proposition}\label{irreducibles}
	Let $\alpha$ be a nonzero nonunit element in $\fz$. Then $\alpha$ is irreducible in $\fz$ if and only if
	\begin{enumerate}
		\vspace{1pt}
		\item $\langle \alpha \rangle$ is a prime ideal in $\fz$ (and hence $\alpha$ is a prime element), or
		\vspace{3pt}
		\item $\langle \alpha \rangle = P_1 P_2$ where $P_1$ and $P_2$ are nonprincipal prime ideals of $\fz$.
	\end{enumerate}
\end{proposition}

\begin{proof}
	($\Rightarrow$) Suppose $\alpha$ is irreducible in $\fz$.  If $\langle \alpha \rangle$ is a prime ideal, then we are done. Assume $\langle \alpha \rangle$ is not a prime ideal. Then by Theorem~\ref{FTIT} there are prime ideals $P_1, \dots, P_k$ of $\fz$ with $\langle \alpha \rangle = P_1 \cdots P_k$ for some $k \geq 2$. Suppose that one of the $P_i$'s is a principal ideal. Without loss of generality, assume that $P_1 = \langle \beta \rangle$ for some prime $\beta$ in $\fz$. Using the class group, $P_2 \cdots P_k = \langle \gamma \rangle$, where $\gamma$ is a nonzero nonunit of $\fz$. Thus, $\langle \alpha \rangle = \langle \beta \rangle \langle \gamma \rangle$ implies that $\alpha = (u \beta) \gamma$ for some unit $u$ of $\fz$. This contradicts the irreducibility of $\alpha$ in $\fz$. Therefore all the $P_i$'s are nonprincipal. Since the class group of $\fz$ is $\z_2$, it follows that $k$ is even. Now suppose that $k > 2$. Using Corollary~\ref{multcorollary} and proceeding in a manner similar to the previous argument, $P_1 P_2 = \langle \beta \rangle$ and $P_3 \cdots P_k = \langle \gamma \rangle$, and again $\alpha = u \beta \gamma$ for some unit $u$, which contradicts the irreducibility of $\alpha$. Hence, either $k=1$ and $\alpha$ is a prime element, or $k=2$.
	
	($\Leftarrow$) If $\langle \alpha \rangle$ is a prime ideal, then $\alpha$ is prime and so irreducible. Then suppose that $\langle \alpha \rangle = P_1 P_2$, where $P_1$ and $P_2$ are nonprincipal prime ideals of $\fz$. Let $\alpha = \beta \gamma$ for some $\beta, \gamma \in \fz$, and assume, without loss of generality, that $\beta$ is a nonzero nonunit of $\fz$. Notice that $\langle \beta \gamma \rangle = \langle \beta \rangle \langle \gamma \rangle = P_1 P_2$. Because $P_1$ and $P_2$ are nonprincipal ideals, $\langle \beta \rangle \notin \{P_1, P_2\}$. As a consequence of Theorem~\ref{FTIT}, we have that $\langle \beta \rangle = P_1 P_2$. This forces $\langle \gamma \rangle = \langle 1 \rangle$, which implies that $\gamma \in \{\pm 1\}$. Thus, $\alpha$ is irreducible.
\end{proof}
\medskip

Let us use  Proposition~\ref{irreducibles} to analyze the factorizations presented in (\ref{2ndzroot5}) at the beginning of the exposition. As the product of any two nonprincipal ideals of $\fz$ is a principal ideal, the decompositions
\begin{align*}
		\langle 6 \rangle
			&= \langle 2 \rangle \langle 3 \rangle = \langle 2,\; 1 + \sqrt{-5} \rangle^2 \langle 3,\; 1 - \sqrt{-5} \rangle \langle 3,\; 1 + \sqrt{-5}\rangle \\
			&= \langle 2,\; 1+\sqrt{-5} \rangle \langle 3,\; 1 + \sqrt{-5} \rangle \langle 2,\; 1 + \sqrt{-5} \rangle \langle 3,\; 1 - \sqrt{-5} \rangle \\
			&= \langle 1 + \sqrt{-5} \rangle \langle 1 - \sqrt{-5} \rangle
\end{align*}
yield that $2 \cdot 3$ and $(1 + \sqrt{-5})(1 - \sqrt{-5})$ are the only two irreducible factorizations of $6$ in $\fz$. Thus, any two irreducible factorizations of $6$ in $\fz$ have the same factorization length. We can take this observation a step further.

\begin{theorem} \label{zroot5hf}
	If $\alpha$ is a nonzero nonunit of $\fz$ and $\beta_1, \dots, \beta_s, \gamma_1, \dots, \gamma_t$ are irreducible elements of $\fz$ with $\alpha = \beta_1 \cdots \beta_s = \gamma_1 \cdots \gamma_t$, then $s=t$.
\end{theorem}

\begin{proof}
	Let $\alpha = \omega_1 \cdots \omega_m$ be a factorization of $\alpha$ in $\fz$ into irreducible elements. By Theorem~\ref{FTIT}, there are unique prime ideals $P_1, \dots, P_k$ in $\fz$ satisfying that $\langle \alpha \rangle = P_1 \cdots P_k$. Suppose that exactly $d$ of these prime ideals are principal and assume, without loss, that $P_i = \langle \alpha_i \rangle$ for all $i \in \{1,\dots,d\}$, where each $\alpha_i$ is prime in $\fz$. Since the class group of $\fz$ is $\z_2$, there exists $n \in \n$ such that $k - d = 2n$. Hence,
	\[
		\langle \alpha \rangle = \left( P_1 \cdots P_d \right) \left(P_{d+1} \cdots P_k \right) = \langle \alpha_1 \cdots \alpha_d \rangle \left(P_{d+1} \cdots P_k \right),
	\]
	and any factorization into irreducibles of $\alpha$ will be of the form $u \alpha_1 \cdots \alpha_d \cdot \beta_1 \cdots \beta_n$, where each ideal $\langle \beta_j \rangle$ is the product of two ideals chosen from $P_{d+1}, \dots, P_k$. As a result, $m = d+n$ and, clearly, $s = t = m$, completing the proof.
\end{proof}
\medskip

Thus, while some elements of $\fz$ admit many factorizations into irreducibles, the number of irreducible factors in any two factorizations of a given element is the same. As we mentioned in the introduction, this phenomenon is called half-factoriality. Since the concept of half-factoriality does not involve the addition of $\fz$, it can also be defined for commutative monoids.

\begin{definition}
	An atomic monoid $M$ is called \emph{half-factorial} if any two factorizations of each nonzero nonunit element of $M$ have the same number of irreducible factors.
\end{definition}
\medskip

Half-factorial domains and monoids have been systematically studied since the 1950's, when Carlitz gave a characterization theorem of half-factorial rings of integers, which generalizes the case of $\fz$ considered in this exposition.

\begin{theorem}[Carlitz \cite{Ca}]
	Let $R$ be the ring of integers in a finite extension field of $\q$. Then $R$ is half-factorial if and only if $R$ has class number less than or equal to two.
\end{theorem}
\medskip

A list of factorization inspired characterizations of class number two can be found in \cite{Ch2}.  In addition, 
a few families of half-factorial domains in a more general setting are presented in \cite{Ki}. We will conclude this paper by exhibiting two simple examples of half-factorial monoids, using the second one to illustrate how to compute the number of factorizations in $\fz$ of a given element.

\begin{example}[Hilbert monoid] \label{ex:Hilbert monoid}
	It is easily seen that
	\[
		H = \{1 + 4k \mid k \in \n_0\}
	\]
	is a multiplicative submonoid of $\n$. The monoid $H$ is called \emph{Hilbert monoid}. It is not hard to verify (Exercise~\ref{ex:irreducibles of Hilbert monoid}) that the irreducible elements of $H$ are
	\begin{enumerate}
		\item the prime numbers $p$ satisfying $p \equiv 1 \pmod{4}$ and
		\vspace{3pt}
		\item $p_1p_2$, where $p_1$ and $p_2$ are prime numbers satisfying $p_i \equiv 3 \pmod{4}$.
	\end{enumerate}
	Therefore every element of $H$ is a product of irreducibles. Also, in the factorization of any element of $H$ into primes, there must be an even number of prime factors congruent to $3$ modulo $4$. Hence, any factorization of an element $x \in H$ comes from pairing the prime factors of $x$ that are congruent to $3$ modulo $4$. This implies that $H$ is half-factorial. For instance, $x = 5^2 \cdot 3^2 \cdot 11 \cdot 13 \cdot 19$ has exactly two factorizations into irreducibles, each of them contains five factors:
	\begin{align*}
		x = 5^2 \cdot 13 \cdot (3^2) \cdot (11 \cdot 19) = 5^2 \cdot 13 \cdot (3 \cdot 11) \cdot (3 \cdot 19).
	\end{align*}
\end{example}
\medskip

\begin{exercise} \label{ex:irreducibles of Hilbert monoid}
	Argue that the irreducible elements of the Hilbert monoid are precisely those described in Example~\ref{ex:Hilbert monoid}.
\end{exercise}
\medskip

\begin{definition}
	Let $p$ be a prime number.
	\begin{enumerate}
		\item We say that $p$ is \emph{inert} if $\langle p \rangle$ is a prime ideal in $\fz$.
		\vspace{3pt}
		\item We say that $p$ is \emph{ramified} if $\langle p \rangle = P^2$ for some prime ideal $P$ in $\fz$.
		\vspace{3pt}
		\item We say that $p$ \emph{splits} if $\langle p \rangle = PP'$ for two distinct prime ideals in $\fz$.
	\end{enumerate}
\end{definition}
\medskip

Prime numbers $p$ can be classified according to the above definition. Indeed, we have seen that $p$ is ramified when $p \in \{2,5\}$. On the other hand, it is also known that $p$ splits if $p \equiv 1, 3, 7, 9 \pmod{20}$ and is inert if $p \not\equiv 1, 3, 7, 9 \pmod{20}$ (except $2$ and $5$). A proof of this result is given in \cite{Ma}.

\begin{example}
	When $n \ge 2$, the submonoid $\mathbb{X}_n$ of the additive monoid $\n_0^{n+1}$ given by
	\[
		\mathbb{X}_n = \{(x_1, \dots, x_{n+1}) \mid x_i \in \mathbb{N}_0 \ \text{ and } \ x_1 + \dots + x_n = x_{n+1}\}
	\]
	is a half-factorial Krull monoid with divisor class group $\z_2$ (see \cite[Section~2]{CKO} for more details). Following \cite{CHR}, we will use $\mathbb{X}_n$ to count the number of distinct factorizations into irreducibles of a given nonzero nonunit $\alpha \in \fz$. Let
	\[
		\langle \alpha \rangle = P_1^{n_1} \cdots P_k^{n_k}Q_1^{m_1} \cdots Q_t^{m_t},
	\]
	where the $P_i$'s are distinct prime ideals in the trivial class ideal of $\fz$, the $Q_j$'s are distinct prime ideals in the nontrivial class ideal of $\fz$, and $m_1 \le \dots \le m_t$. Then the desired number of factorizations $\eta(\alpha)$ of $\alpha$ in $\fz$ is given by
	\[
		\eta(\alpha) = \eta_{\mathbb{X}_t}\bigg( m_1, \dots, m_t, \frac{m_1 + \dots + m_t}2 \bigg),
	\]
	which, when $t = 3$, can be computed by the formula
	\[
		\eta_{\mathbb{X}_3}(x_1, x_2, x_3, x_4) = \sum_{j=0}^{\lfloor x_1/2 \rfloor} \sum_{k=0}^{x_1 - 2j} \bigg(\bigg\lfloor \frac{\min\{x_2 - k, x_3 - x_1 + 2j + k\}}{2} \bigg\rfloor + 1\bigg).
	\]
	For instance, let us find how many factorizations $1980 = 2^2 \cdot 3^2 \cdot 5 \cdot 11$ has in $\fz$. We have seen that $5$ ramifies as $\langle 5 \rangle = P_1^2$, where $P_1$ is principal. As $11$ is inert, $P_2 = \langle 11 \rangle$ is prime. In addition, $3$ splits as $\langle 3 \rangle = Q_1 Q_2$, where $Q_1$ and $Q_2$ are nonprincipal. Finally, $2$ ramifies as $\langle 2 \rangle = Q_3^2$, where $Q_3$ is nonprincipal. Therefore one has that $\langle 1980 \rangle = P_1^2 P_2 Q_1^2Q_2^2Q_3^4$, and so
	\[
		\eta(1980) = \eta_{\mathbb{X}_3}(2,2,4,4) = \sum_{j=0}^1 \sum_{k=0}^{2-2j} \bigg(\bigg\lfloor \frac{\min\{2 - k, 2 + 2j + k\}}{2} \bigg\rfloor + 1\bigg) = 6.
	\]
\end{example}
\medskip

\section*{Acknowledgements}

It is a pleasure for the authors to thank the referee, whose helpful suggestions vastly improved the final version of this paper.

%%%%%%%%%%%%%%%%%%%%%%%% References %%%%%%%%%%%%%%%%%%%%%%%%%%%%%%%%%%

\end{document}